\documentclass{colt2016} 

\title[The low-rank approach for certain semidefinite programs]{On the low-rank approach for semidefinite programs arising in synchronization and community detection}

\usepackage{times}

 \coltauthor{\Name{Afonso S. Bandeira} \Email{bandeira@mit.edu}\\
 \addr Department of Mathematics, Massachusetts Institute of Technology, Cambridge, Massachusetts, USA.
 \AND
 \Name{Nicolas Boumal} \Email{nboumal@math.princeton.edu}\\
 \addr Mathematics Department, Princeton University, Princeton, New Jersey, USA.
  \AND
 \Name{Vladislav Voroninski} \Email{vvlad@math.mit.edu}\\
 \addr Department of Mathematics, Massachusetts Institute of Technology, Cambridge, Massachusetts, USA.
 }

\newcommand{\ER}{{Erd\H{o}s-R\'enyi }}

\newcommand{\transpose}{^\top\! }

\newcommand{\inner}[2]{\left\langle{#1},{#2}\right\rangle}
\newcommand{\innersmall}[2]{\langle{#1},{#2}\rangle}

\newcommand{\trace}{\mathrm{Tr}}
\newcommand{\Trace}{\mathrm{Tr}}
\newcommand{\spann}{\mathrm{span}}

\newcommand{\Rnn}{{\mathbb{R}^{n\times n}}}
\newcommand{\Rnp}{{\mathbb{R}^{n\times p}}}

\newcommand{\complex}{{\mathbb{C}}}
\newcommand{\Rn}{{\mathbb{R}^n}}
\newcommand{\Rm}{{\mathbb{R}^m}}

\newcommand{\diag}{\mathrm{diag}}

\newcommand{\calA}{\mathcal{A}}

\newcommand{\rank}{\operatorname{rank}}
\newcommand{\nulll}{\operatorname{null}}

\newcommand{\1}{\mathbf{1}}

\newcommand{\ZZ}{\mathbb{Z}}
\newcommand{\RR}{\mathbb{R}}
\newcommand{\EE}{\mathbb{E}}
\newcommand{\GGG}{\mathcal{G}}
\newcommand{\YYY}{\mathcal{Y}}
\newcommand{\NNN}{\mathcal{N}}
\newcommand{\eps}{\varepsilon}

\newcommand{\ddiag}{\mathrm{ddiag}}

\newcommand{\defeq}{\mathrel{\mathop:}=}

\begin{document}

\maketitle

\begin{abstract}
To address difficult optimization problems, convex relaxations based on semidefinite programming are now common place in many fields. Although solvable in polynomial time, large semidefinite programs tend to be computationally challenging. Over a decade ago, exploiting the fact that in many applications of interest the desired solutions are low rank, Burer and Monteiro proposed a heuristic to solve such semidefinite programs by restricting the search space to low-rank matrices. The accompanying theory does not explain the extent of the empirical success. We focus on Synchronization and Community Detection problems and provide theoretical guarantees shedding light on the remarkable efficiency of this heuristic. 
\end{abstract}


\section{Introduction}

Estimation problems in many fields, including signal processing, statistics and machine learning, are formulated as intractable optimization problems. A now popular technique to attempt solving some of these problems is to replace the difficult optimization problem by a surrogate tractable convex problem and take advantage of existing machinery for convex optimization. In many instances, the surrogate problem is obtained by considering a larger, convex feasible space, and thus it is often called a \textit{convex relaxation}. Relaxations based on semidefinite programming are among the most popular. 

We will focus mostly on a certain class of problems, namely, \emph{synchronization problems} on graphs (\cite{singer2010angular,bandeira2015nonuniquegames}). Synchronization problems consist in estimating $n$ group elements $g_1,\dots,g_n$ from information about their offsets $g_i^{}g_j^{-1}$. A simple instance of this class is $\ZZ_2$-Synchronization (\cite{abbe2014decoding}), corresponding to the group on two elements. In that case, the task is to estimate $n$ bits $x_1,\dots,x_n\in\{\pm1\}$ from noisy measurements of $x_ix_j$.

The problem of Community Detection in the Binary Stochastic Block Model (SBM) can also be regarded as an instance of Synchronization over $\ZZ_2$. SBM on two communities (also known as \emph{planted partition}) is a model of a random graph $\GGG(n,p,q)$ on $n=2m$ nodes split evenly in two communities, identified by $\left\{\pm1\right\}$ node labels. Edges for this graph are drawn randomly and independently, where each pair of nodes in the same community gets connected with probability $p$ and in different communities with probability $q$. The task is to estimate the partition, given an observation of the graph. This fascinating problem was the target of much study in the last few years, including identification of remarkable phase transitions on the values of $p$ and $q$ and possibility of recovering the unknown labels (\cite{Decelle_SBM,Mossel_SBM1,Mossel_SBM2,Massoulie_SBM,abbe2014exact,Mossel_SBM3_exact}).

While computing the maximum likelihood estimator (MLE) for either of the problems above is known to be computationally intractable, several heuristics have been proposed and studied. A particularly successful approach is based on Semidefinite Programming. Following the seminal work of \cite{goemans1995maxcut} in the context of the Max-Cut problem, the maximum likelihood estimation problem (originally an intractable optimization problem in $n$ node variables) is relaxed to a semidefinite program (SDP) (a tractable problem on $n^2$ variables). Importantly, the solution to the SDP is only guaranteed to correspond to the solution of the original problem if it has rank~1. Enforcing this constraint explicitly would render the problem intractable. Remarkably, in some regimes, it is known that the solution of this semidefinite program is naturally of rank 1, and that furthermore this solution allows to identify the true labels (\cite{abbe2014decoding,bandeira2014tightness,abbe2014exact,bandeira2015laplacian,Hajek_et_al_SBM_SDP}).

While SDP's are known to be solvable up to arbitrary precision in polynomial time (\cite{nesterov2004introductory}), the increase in dimension due to the lifting technique (the relaxation) and the semidefiniteness constraint render solving them rather slow in practice. This is mainly because intermediate iterates involve matrix decompositions of dense, full-rank matrices of size $n$. On the other hand, in certain regimes, the optimal solution is expected to have low rank. Indeed, in the problems studied here, the solution being rank 1 coincides precisely with it corresponding to the desirable MLE. It is then natural to attempt to reach the solution via a sequence of similarly simple objects instead.

The most popular and successful such low-rank approach is SDPLR, as proposed in~\citep{sdplr,burer2005local} (also in a particular form in \citep{burer2002rank}). In a nutshell, the idea is to restrict the search space to matrices of bounded rank. While, after adding this rank constraint, the optimization problem is no longer convex (and it is hence unclear whether it is tractable or not),
this approach is empirically successful for a variety of instances~\citep{sdplr,journee2010low,bandeira2014tightness,boumal2015staircase}.

The original SDPLR papers come with general theory supporting the fact that relaxing the rank up to about $\sqrt{2n}$ might work well. Yet, in practice, for well-behaved SDP's which admit a solution of rank 1, it is often seen that relaxing the rank merely to 2 works fantastically. To date, there was no satisfactory explanation for this nonconvex success. This paper provides the first such guarantee, in the context of synchronization over $\ZZ_2$ and community detection. More precisely, we show that, in certain noise regimes, there are no spurious second-order critical points for the
rank-2 constrained problem, while in other (more inclusive) regimes we show that all such points need to correlate non-trivially with the ground truth. Second-order critical points (that is, points which satisfy second-order necessary optimality conditions) can be computed in our context~\citep{boumal2016complexity}.

In this paper, we focus on a selection of well-studied problems, keeping in mind that the proposed analysis has the potential to generalize to many problems where semidefinite relaxations are successful yet demanding to solve. A more general setting will be the focus of a future publication.

It is also worth noting that semidefinite relaxations have been observed to work well on problems for which other standard heuristics seem to not perform well. One relevant example is the Multireference Alignment problem~(\cite{Bandeira_Charikar_Singer_Zhu_Alignment}). Furthermore, low-rank semidefinite programming has the potential to yield a posteriori certifiably correct algorithms~(\cite{Bandeira_PCC}) and is known to be robust to certain monotone adversary models~(\cite{Feige_Kilian_bisection_01,Moitra_SBMadversary}).

\subsection*{Notation}

Given a matrix $M$ and its vector of singular values $\nu$, we write $\|M\| = \|\nu\|_\infty$ for its operator norm (largest singular value), $\|M\|_F = \|\nu\|_2$ for its Frobenius norm and $\|M\|_* = \|\nu\|_1$ for its nuclear norm (sum of singular values). The operator $\ddiag \colon \Rnn \to \Rnn$ sets all off-diagonal entries of a matrix to zero. The Hadamard or entry-wise product is written $\circ$.

\section{The $\ZZ_2$ Synchronization problem}

The goal of $\ZZ_2$ Synchronization is to estimate labels $z\in\{\pm 1\}^n$ from noisy pairwise measurements $Y_{ij} = z_iz_j + \sigma W_{ij}$, where $W_{i>j} \stackrel{i.i.d.}{\sim} \NNN(0,1)$ and $W_{ij}=W_{ji}$, $W_{ii} = 0$. In matrix notation,
\begin{equation}\label{eq:noisemodel:sigma}
 Y = zz^T + \sigma W.
\end{equation}
Because only $zz^T$ is measured, the labels can only be estimated up to a global sign flip. One can also consider more realistic noise models on which $Y_{ij}$ also takes binary values (\cite{abbe2014decoding}) but, for the sake of simplicity, we will consider Gaussian noise. Since $\|zz^T\| = n$ and $\|\sigma W\|$ concentrates around a constant multiple of $\sigma \sqrt{n}$ (in operator norm), a natural way to parametrize the signal-to-noise ratio is by taking $\lambda = \frac{\sqrt{n}}{\sigma}$. For this reason, some authors consider the scaled model
\begin{equation}\label{eq:noisemodel:lambda}
 \YYY = \frac{\lambda}{n}Y =  \frac{\lambda}{n}zz^T + \frac1{\sqrt{n}}W.
\end{equation}
The problems of recovering $z$ from $Y$ or $\YYY$ are clearly equivalent. The former choice of notation has been used, for example, in~\citep{bandeira2014tightness,bandeira2015laplacian} and has the advantage of highlighting the fact that the observation is an entry-wise noisy version of the ground truth, while the latter has been used in~\citep{javanmard2015phase} and has the advantage of highlighting the spike model interpretation of the problem and making a more transparent connection with the well-studied spike model in random matrix theory and the BBP transition phenomenon~\citep{BBP_MC_BBP_2005,Feral_Peche_BBPWigner} (this connection had already been made in~\citep{singer2010angular}). For the sake of completeness, we will state our results in both notation choices.

The most natural approach to recovering $z$ is to consider the MLE, which solves
\begin{equation}
\min_{x\in\{\pm 1\}^n} \sum_{i,j=1}^n \left( Y_{ij} - x_ix_j  \right)^2.
\end{equation}
It is readily seen that solutions of this problem are the same as those of
\begin{equation}\label{eq:littleGrothendieck}
\max_{x\in\{\pm 1\}^n} x^TYx,
\end{equation}
which is known to the \emph{NP-hard} in general. In fact, when $Y\succeq 0$ this is known as the little Grothendieck problem~\citep{briet2014grothendieck,bandeira2013approximating}, and, when $Y$ corresponds to the Laplacian of a graph, it corresponds to the Max-Cut problem~\citep{goemans1995maxcut}.

Following the now standard lifting technique (dating back at least to~\cite{goemans1995maxcut} ), we set a new variable $X = xx^T$ and write the equivalent formulation
\begin{equation}\label{SDP:rank1}
\begin{array}{ll}
\max_X & \trace\left( YX \right) \\
\text{s.t.} & X_{ii} = 1, \text{ for } 1\leq i \leq n\\ 
 & X \succeq 0 \\
 & \rank(X) = 1.
\end{array}
\end{equation}
Problems~\eqref{eq:littleGrothendieck} and~\eqref{SDP:rank1} are equivalent.
One arrives at a tractable SDP formulation by dropping the problematic rank constraint:
\begin{equation}\label{SDP:fullrank}
\begin{array}{ll}
\max & \trace\left( YX \right) \\
\text{s.t.} & X_{ii} = 1, \text{ for } 1\leq i \leq n\\ 
 & X \succeq 0.
\end{array}
\end{equation}
By construction, problem \eqref{SDP:fullrank} has the same solutions if $Y$ is replaced with $\YYY$ (but it will have a different optimal value, since $Y$ and $\YYY$ are scaled versions of each other.)

Recall the model~\eqref{eq:noisemodel:sigma} (or equivalently~\eqref{eq:noisemodel:lambda}). It is known that, as long as $\sigma < \sqrt{\frac{n}{2\log n}}$, or equivalently $\lambda > \sqrt{2\log n}$, with high probability, the solution $X$ of~\eqref{SDP:fullrank} is unique and corresponds exactly to the ground truth $X=zz^T$. We refer to this phenomenon as \emph{exact recovery}. It is also known that on the other side of this threshold exact recovery is information-theoretically impossible, showcasing the efficiency of semidefinite relaxations for this type of task~\citep{bandeira2014tightness,bandeira2015laplacian}.

While exact recovery is impossible for constant $\lambda$ (or, equivalently, $\sigma \sim \sqrt{n}$), simply taking the top eigenvector of $\YYY$ is known to produce an estimator that correlates non-trivially with the ground truth, precisely when $\lambda >1$ (\cite{Feral_Peche_BBPWigner,singer2010angular}). This phenomenon is often referred to as the BBP transition (\cite{BBP_MC_BBP_2005}). This motivates the question of whether~\eqref{SDP:fullrank} gives meaningful results for the constant $\lambda$ regime~(\cite{Montanari_SDPdetectionSBM,javanmard2015phase,Guedon_SBMgrothendieck}). In the context of community detection, this question was first addressed by~\cite{Guedon_SBMgrothendieck}. In~\citep{Montanari_SDPdetectionSBM}, a phase transition is shown to exist at $\lambda=1$: similarly to what happens with the top eigenvalue of $\YYY$ (\cite{Feral_Peche_BBPWigner}). For $\lambda<1$, the value of~\eqref{SDP:fullrank} with cost $\YYY$ converges (in probability) to $2$, and for $\lambda>1$ its limit is strictly larger than $2$. \cite{javanmard2015phase} give fascinating predictions, based on non-rigorous statistical mechanics tools, for the behavior of both~\eqref{SDP:fullrank} and~\eqref{eq:littleGrothendieck} as a function of $\lambda$.

Another, related, type of synchronization problem is \emph{phase synchronization}, where one estimates $z\in\complex^n$ with $|z_i|=1$ for all $i$. This is a direct complex analogue of $\ZZ_2$ synchronization. The SDP relaxation works similarly in the complex field. See~\citep{singer2010angular} for a first analysis of the SDP relaxation, see~\citep{bandeira2014tightness} for a proof of tightness of the SDP relaxation in a similar regime as here and, still in the same regime, see~\citep{boumal2016nonconvexphase} for a proof that the nonconvex problem has no spurious local optima either. A key difference with the present paper is that phase synchronization is a continuous problem, whereas $\ZZ_2$ synchronization is discrete. This explains why, in the present setting, the constraint $\rank(X)=1$ has to be relaxed at least to $\rank(X)=2$ (to connect the search space), whereas in the former paper, the rank is not relaxed at all.

\subsection{The Burer--Monteiro Approach}

In transiting from~\eqref{SDP:rank1} to~\eqref{SDP:fullrank}, one effectively replaces the constraint $\rank(X) \leq 1$ by the vacuous constraint $\rank(X) \leq n$, thus going from a combinatorial problem to a tractable but high-dimensional SDP. One of the insights of~\cite{sdplr,burer2005local} is that relaxing the rank only partially, as $\rank(X)\leq p$ for variable $p$, gives access to a family of low-dimensional (but nonconvex) nonlinear optimization problems, which can be put to good use to understand both~\eqref{SDP:rank1} and~\eqref{SDP:fullrank}.

In general, given an SDP of the form
\begin{align}
	\max_{X} \Trace(YX) \quad \textrm{ s.t. } \quad \calA(X) = b, X \succeq 0,
\end{align}
where $\calA$ is a linear operator from the symmetric matrices of size $n$ to $\Rm$,
the SDPLR algorithm (sometimes referred to as the Burer--Monteiro approach) consists in parameterizing $X$ as $X = QQ^T$, where $Q$ lives in $\Rnp$. In so doing, $X \succeq 0$ is naturally enforced. This yields the following nonlinear optimization problem:
\begin{align}
\max_{Q\in\Rnp} \Trace(Q^TYQ) \quad \textrm{ s.t. } \quad \calA(QQ^T) = b,
\end{align}
where both the cost and the constraints are quadratic in $Q$. This is typically nonconvex. Both problems are equivalent, up to the additional constraint $\rank(X) \leq p$ forced in the nonlinear program. Of course, if the SDP admits a solution of rank at most $p$, then the two problems attain the same optimal value. When the search space is compact, this is known to happen as soon as $p(p+1)/2 \geq m$ ($m$ is the number of constraints), as per general results put forth in~\citep{shapiro1982rank,barvinok1995problems,pataki1998rank}. \cite{burer2005local} show that rank deficient local optima of the nonlinear program are globally optimal.\footnote{This renders these methods potentially a posteriori certifiable (\cite{Bandeira_PCC}).} See also Lemma~\ref{lem:rankdeficientglobalopt}.

The latter observations motivate the algorithm SDPLR~\citep{sdplr}, where the nonlinear program is tackled in lieu of the SDP, using classical nonlinear optimization algorithms such as the augmented Lagrangian method. The above discussion suggests setting $p \approx \sqrt{2m}$. If the local method converges to a rank-deficient local optimum (which is often the case in practice but is not satisfactorily explained in theory), then we have found a global optimum. The advantage (compared to the SDP) is that the search space has lower dimension.


In practice, if we are optimistic about our chances, we can set $p$ as low as $p=2$ for example~(\cite{burer2002rank}). If the SDP has a solution of rank 1, then the corresponding nonlinear program has local optima (actually, global optima) of rank 1 as well, thus being rank deficient, which allows certifying their optimality. In practice, this is seen to work remarkably well for the problems considered in this paper (in certain noise regimes), as had already been observed in other papers as well. It is not obvious why this is so, as the rank restriction could (in general) spawn a large number of spurious local optima.

Here, for the first time, we provide a proof that, at $p=2$ and in favorable regimes, all local optima are global optima, which explains why local methods behave appropriately. In fact, we even show that saddle points can always be escaped using solely second-order derivatives.

Local algorithms such as the augmented Lagrangian method are useful to tackle generic constrained nonlinear programs. The SDP's under scrutiny, though, have rather special structure. In particular, they are such that the search space in $Q$ is a smooth manifold already for $p\geq 2$. As originally advocated by~\cite{journee2010low}, this suggests solving the nonlinear program using techniques from optimization on manifolds~\citep{AMS08}. The latter are indeed better suited to fully leverage the special geometry of these problems. See also~\citep{wen2013orthogonality,boumal2015staircase} for further numerical and theoretical investigation.
See~\citep{boumal2016complexity} for a Riemannian version of the trust-region method which globally converges to global optima owing to the special properties of our problem, with global rates of convergence.

As described above, here we consider the rank constrained problem as follows (\cite{burer2002rank}):
\begin{equation}\label{rank2}
\begin{array}{ll}
\max & \trace\left( YX \right) \\
\text{s.t.} & X_{ii} = 1, \text{ for } 1\leq i \leq n\\ 
 & X \succeq 0 \\
 & \rank(X) \leq 2. 
\end{array}
\end{equation}
By taking $X = QQ^T$ with $Q\in\RR^{n\times 2}$, one can reformulate~\eqref{rank2} as:
\begin{equation}\label{rank2BM}
\begin{array}{ll}
\max & \trace\left( Q^TYQ \right) \\
 & \text{s.t.} \left\|Q_{i:}\right\|^2 = 1, \text{ for } 1\leq i \leq n \\
 & Q \in \RR^{n \times 2},  
\end{array}
\end{equation}
where $Q_{i:}$ denotes the $i$th row of $Q$. Note the geometry here: the search space is a product of circles.


We now present our main results in the realm of $\ZZ_2$ Synchronization.

\begin{theorem}\label{thm:Z2Synch:partialrecovery}
Consider model~\eqref{eq:noisemodel:lambda} (or equivalently~\eqref{eq:noisemodel:sigma}). If $\lambda > 8$ (or equivalently $\sigma < \frac18\sqrt{n}$), then, with high probability, all second-order critical points $Q$ of~\eqref{rank2BM} correlate non-trivially with the ground truth $z$ in the sense that, for every such $\lambda$, there exists $\eps$ such that
\begin{equation}\label{eq:nontrivialcorrelation_0}
\frac1{n} \left\| Q^Tz\right\|_2 \geq \eps.
\end{equation}
\end{theorem}
\begin{proof}
All the interesting ingredients of the proof are in Section~\ref{section:mainresults}. The claim will follow from Lemma~\ref{lemma:Xcorrelates11T}, noting that $\left\| Q^Tz\right\|_2^2 = \left\langle QQ^T,zz^T \right\rangle$, the behavior of~\eqref{rank2BM} is unchanged by changing the diagonal values of $\YYY$, and the fact that for any $\delta>0$, with high probabily, $\frac1n\mathrm{SDP}(W-\ddiag(W))\leq \|W-\ddiag(W)\| \leq (2+\delta)\sqrt{n}$ (see Lemma~\ref{lemma:Wignermatrices} and eq.~\eqref{SDP:eq:inequality}).
\end{proof}

\begin{remark}[Rounding]
In Theorem~\eqref{thm:Z2Synch:partialrecovery} we ask for non-trivial correlation in the form of~\eqref{eq:nontrivialcorrelation_0}, but it would also be natural to ask for an estimator $\hat{z}\in\RR^n$ (or even $\hat{z}\in\{\pm1\}^n$) that exhibits non trivial correlation with $z$. We note that if we take $x$ and $y$ to be the columns of $Q$ satisfying~\eqref{eq:nontrivialcorrelation_0}, then a random linear combination $g_1x+g_2y$, where $g_1$ and $g_2$ are independent standard Gaussian variables, can be shown to be likely to produce meaningful correlation estimators. Indeed, $\EE\{ \|Qg\|^2 \} = \|Q\|_F^2 = n$ and $\EE\{  \inner{z}{Qg}^2 \} = \|Q^Tz\|_2^2$. We also direct the reader to Section~4 of~\citep{Montanari_SDPdetectionSBM} for techniques to construct meaningful $\{\pm1,0\}^n$ estimators from solutions of the SDP. We note furthermore that if $\eps$ is large in~\eqref{eq:nontrivialcorrelation_0} (as will be the case in Theorem~\ref{lem:Z2Synch:almostexactrecovery}, for example), then constructing such estimators becomes considerably easier.
\end{remark}

\begin{theorem}\label{lem:Z2Synch:almostexactrecovery}
Consider model~\eqref{eq:noisemodel:lambda} (or equivalently~\eqref{eq:noisemodel:sigma}). If $\lambda > 16$ (or equivalently $\sigma < \frac1{16}\sqrt{n}$), then for any $\eps>0$, with high probability,
all second-order critical points $Q$ of~\eqref{rank2BM} satisfy
\begin{equation}\label{eq:nontrivialcorrelation_0}
\frac1{n} \left\| Q^Tz\right\|_2 \geq 1 - (1+\eps)\frac{16}{\lambda} =  1 - (1+\eps)\frac{16\sigma}{\sqrt{n}}.
\end{equation}
\end{theorem}
\begin{proof}
Again, all the interesting ingredients of the proof are in Section~\ref{section:mainresults}, as this will follow directly from Lemma~\ref{lemma:Qcorrelates1} and the considerations in the proof of Theorem~\ref{thm:Z2Synch:partialrecovery}.
\end{proof}

\begin{theorem}\label{lem:Z2Synch:exactrecovery}
Consider model~\eqref{eq:noisemodel:lambda} (or equivalently~\eqref{eq:noisemodel:sigma}). There is a universal constant $C$ such that: If $\lambda \geq Cn^{\frac13}$ (or equivalently $\sigma \leq \frac1{C}n^{\frac16}$) then, with high probability, all second-order critical points $Q$ of~\eqref{rank2BM} are optimal and correspond to the ground truth, meaning $QQ^T = zz^T$.
\end{theorem}
\begin{proof}
	This follows from Theorem~\ref{thm:maxcutrktwo} and the high probability bounds in Lemma~\ref{lemma:Wignermatrices}.
\end{proof}




Note that there are no guarantees in general that nonlinear optimization algorithms (such as the augmented Lagrangian method) converge to local optima. Yet, we know that only local optima are stable fixed points for such methods; and that this is true regardless of initialization. Hence, one perspective is that the contribution of this paper is to show that all stable fixed points of local methods---without the need for a good initial guess---are global optima (in the given regime). As we mentioned earlier though, because in our case second-order necessary optimality conditions are sufficient for global optimality, results in~\citep{boumal2016complexity} show that the Riemannian trust-region method converges to global optimizers, regardless of initialization, with global rates of convergence.

\subsection{Community Detection in the Binary Stochastic Block Model}

The Stochastic Block Model (SBM) on two communities is a random graph model
$\GGG(n,p,q)$ on $n=2m$ nodes. The nodes are divided evenly in two
communities, identified by a vector of labels $g\in\{\pm1\}^n$ satisfying $g^T\1=0$.
Edges on this graph are drawn independently at
random. A pair of nodes in the same community is connected by an edge
with probability $p$; nodes in different communities with probability
$q$. We focus on the case $p>q$. This means that the adjacency
matrix $A$ of this graph has the following distribution:
\[
 A_{ij} = \left\{ \begin{array}{ll} 1 & \text{with probability } p \\
0 & \text{with probability } 1-p, \end{array} \right.
\]
if $g_ig_j = 1$ (thus including the diagonal), and
\[
 A_{ij} = \left\{ \begin{array}{ll} 1 & \text{with probability } q \\
0 & \text{with probability } 1-q, \end{array} \right.
\]
if $g_ig_j = -1$. 

The goal of community detection is to recover the labels $g$ from a
realization $G\sim\GGG(n,p,q)$. Roughly half a decade
ago,~\cite{Decelle_SBM} conjectured a remarkable phase transition in
the constant average--degree regime: if $p = \frac{a}n$ and $q =
\frac{b}n$ with $a>b$ constants,~\cite{Decelle_SBM} conjectured that
as long as
\begin{equation}\label{eq:condition:partialrecovery:SBM}
\lambda(a,b) \defeq \frac{a-b}{\sqrt{2(a+b)}} > 1,
\end{equation}
it is possible to construct an estimator that, with high probability,
correlates with the true partition, whereas below this threshold that would be impossible. This conjecture was proven in a series of
papers~\citep{Mossel_SBM1,Mossel_SBM2,Massoulie_SBM}.

Another natural question is to understand when it is possible to
exactly recover $g$ (or $-g$). A similar phase transition was
established in~\citep{abbe2014exact} and~\citep{Mossel_SBM3_exact}: if
$p = \alpha\frac{\log n}n$ and $q = \beta\frac{\log n}n$ with
$\alpha>\beta$ constants, then as long as
\begin{equation}\label{eq:conditionsalphabetaSBMcor}
 \sqrt{\alpha} - \sqrt{\beta} > \sqrt{2},
\end{equation}
the MLE for the partition, solution of
\begin{equation}\label{eq:MLE_SBM}
\begin{array}{ll}
 \max & x^TAx \\
 \text{ s.t. } & x \in \{ \pm 1 \}^n \\
 & x^T \1 = 0,
\end{array}
\end{equation}
coincides exactly with the partition, with high probability; and,
moreover, below this threshold it is information-theoretically
impossible to recover the partition.

\cite{abbe2014exact} studied a semidefinite relaxation
of~\eqref{eq:MLE_SBM} analogous to~\eqref{SDP:fullrank}, and
conjectured that its solution coincides with the ground truth (that is,
the SDP achieves exact recovery) at the information-theoretical
limit~\eqref{eq:conditionsalphabetaSBMcor}. This was confirmed
independently in~\citep{Hajek_et_al_SBM_SDP}
and~\citep{bandeira2015laplacian}. The performance of the semidefinite
relaxation on the constant average--degree regime has also been target
of recent work (\cite{Guedon_SBMgrothendieck,Montanari_SDPdetectionSBM,javanmard2015phase}).
\cite{Guedon_SBMgrothendieck} showed that when the left hand side
of~\eqref{eq:condition:partialrecovery:SBM} is a sufficiently large
constant, the solution of the SDP correlates nontrivially with the
partition. This was improved by~\cite{Montanari_SDPdetectionSBM} who
showed that, for every $\epsilon>0$, and large enough average degree,
$\lambda(a,b) > 1+\epsilon$ suffices. 
\cite{javanmard2015phase} give precise predictions for the
behavior of the SDP in this regime, albeit using non-rigorous
techniques. Interestingly, the results in~\citep{Moitra_SBMadversary}
suggest that there may be values of $a$ and $b$
satisfying~\eqref{eq:condition:partialrecovery:SBM} for which the SDP
does not, with high probability, provide meaningful solutions.



We will now write $A$ in a form close to~\eqref{eq:noisemodel:lambda} to help illustrate how community detection in the SBM is closely related to $\ZZ_2$ Synchronization. We start by noting that 
\[
\EE A = \frac{p+q}2\1\1^T + \frac{p-q}2gg^T.
\]
Since $\1$ is a fixed vector, we can consider $A^\natural$ defined as
\[
A^\natural = A - \frac{p+q}2\1\1^T.
\]
It is readily seen that, because of the balanced partition constraint ($x^T\1=0$), replacing $A$ by $A^\natural$ does not affect the solution of the MLE~\eqref{eq:MLE_SBM}. The benefit is that now, since $A^\natural$ no longer has a bias in the direction of $\1\1^T$, we will remove the extra constraint and focus on trying to understand the efficiency of the simpler program:
\begin{equation}\label{eq:MLE_SBM_nobalancing}
\begin{array}{ll}
\max & x^T A^\natural x \\
\text{s.t.} & x \in \{\pm1\}^n,
\end{array}
\end{equation}
and its natural SDP relaxation
\begin{equation}\label{SDP:fullrank_SBM}
\begin{array}{ll}
\max & \trace\left( A^\natural X \right) \\
\text{s.t.} & X_{ii} = 1, \text{ for } 1\leq i \leq n\\ 
 & X \succeq 0.
\end{array}
\end{equation}
Following the Burer--Monteiro approach, we consider the natural analogue to~\eqref{rank2BM}:
\begin{equation}\label{rank2BM_SBM}
\begin{array}{ll}
\max & \trace\left( Q^TA^\natural Q \right) \\
 & \text{s.t.} \left\|Q_{i:}\right\|^2 = 1, \text{ for } 1\leq i \leq n \\
 & Q \in \RR^{n \times 2}. 
\end{array}
\end{equation}
We further write
\begin{equation}\label{eq:def:E}
A^\natural = \frac{p-q}{2}gg^T + \frac{(p-q)n}{2}E + \frac{(p-q)n}{2}D,
\end{equation}
where $D$ is a diaognal matrix, and $E$ is a zero-diagonal centered random matrix (we choose to normalize by the average degree $\frac{(p-q)n}{2}$ to be compatible with the notation in~\citep{Montanari_SDPdetectionSBM}, whose statements about $E$ we will use). We note also that the diagonal matrix $D$ does not affect the solutions of~\eqref{eq:MLE_SBM},~\eqref{SDP:fullrank_SBM}, or~\eqref{rank2BM_SBM}. This means in particular that, in the constant average--degree regime ($p=\frac{a}n$ and $q=\frac{b}{n}$),
\[
\sqrt{\frac2{a+b}}A^\natural - D = \frac{\lambda(a,b)}n gg^T + E,
\]
for $\lambda(a,b)$ defined in~\eqref{eq:condition:partialrecovery:SBM}. This illustrates the parallelism with~\eqref{eq:noisemodel:lambda}. The main difficulty with this setting is that the spectral norm of $E$ is known not to be bounded (as $n\to\infty$), with high probability (similarly to how adjacency matrices of constant average--degree \ER graphs typically have a fluctuation around their mean whose spectral norm is not bounded, see for example~\citep{Krivelevich_Sudakov_ER,Montanari_SDPdetectionSBM}). For this reason, we will focus on another quantity. 

Following the notation in~\citep{Montanari_SDPdetectionSBM}, given a matrix $M$ we define $\mathrm{SDP}(M)$ as
\begin{equation}\label{def:SDP:eq}
\begin{array}{lll}
\mathrm{SDP}(M)
= &\max & \trace\left( M X \right) \\
& \text{s.t.} & X_{ii} = 1, \text{ for } 1\leq i \leq n\\ 
 & & X \succeq 0,
\end{array}
\end{equation}
and note that, since in the search space $\|X\|_\ast = n$, we have
\begin{equation}\label{SDP:eq:inequality}
\mathrm{SDP}(M) \leq n\|M\|.
\end{equation}
(Use H\"older's inequality: $\inner{M}{X} \leq \|M\|\|X\|_*$ and $\|X\|_* = \trace(X) = n$ since $X\succeq 0$.)

Fortunately, Lemma~\ref{lem:optimality}, which is crucial to establish non-trivial correlation guarantees, does not depend on $\|E\|$, which is not bounded, but rather depends only on $\frac1n\mathrm{SDP}(E)$, which is known to be bounded~\citep{Montanari_SDPdetectionSBM} (see Lemma~\ref{lemma:sbm:SDPE} for details). Indeed, using Lemmas~\ref{lem:optimality} and~\ref{lemma:sbm:SDPE} we immediately get the following guarantee.

\begin{theorem}
 Consider the Stochastic Block Model on two communities described above in the constant average--degree regime $p=\frac{a}n$ and $q=\frac{b}n$. Let $d$ denote the average degree parameter
 \[
  d = \frac{a+b}2.
 \]
Then, for any $\delta>0$ there exists $d_0$ such that if $d>d_0$ and $\lambda(a,b) > 8+\delta$ (see~\eqref{eq:condition:partialrecovery:SBM}) then there exists $\eps>0$ such that: with high probability, all second-order critical points $Q$ of~\eqref{rank2BM_SBM} correlate non-trivially with the ground truth partition $g$ in the following sense.
\begin{equation}\label{eq:nontrivialcorrelation_0}
\frac1{n} \left\| Q^Tg\right\|_2 \geq \eps.
\end{equation}
\end{theorem}
\begin{proof}
The proof is analogous to Theorem~\ref{thm:Z2Synch:partialrecovery}, but based on the bound on $\mathrm{SDP(E)}$ in Lemma~\ref{lemma:sbm:SDPE}.
\end{proof}

Further controlling $\|E\|$ and $\|Eg\|_\infty$ (Lemmas~\ref{lemma:sbm:spectralnormE} and~\ref{lemma:sbm:infnormE}) and using Lemma~\ref{thm:maxcutrktwo} we can also obtain a (suboptimal) exact recovery guarantee. It is useful to define the parameter (analogous to~\eqref{eq:condition:partialrecovery:SBM}):
\[
 \tilde{\lambda}(p,q) \defeq \frac{p-q}{\sqrt{2(p+q)}}\sqrt{n}.
\]
Note that if $p=\frac{a}n$ and $q=\frac{b}n$ then $\tilde{\lambda}(p,q) =  \lambda(a,b)$.
\begin{theorem}
  Consider the Stochastic Block Model on two communities described above. There exists a universal constant $c$ such that, as long as
  \[
  \tilde{\lambda}(p,q) \geq cn^{1/3},
  \]
then, with high probability, all second-order critical points $Q$ of~\eqref{rank2BM_SBM} correspond to the ground truth partition, meaning $QQ^T = gg^T$.
\end{theorem}

\begin{proof}
If $\frac{p-q}{\sqrt{2(p+q)}}\sqrt{n} \geq cn^{1/3}$, then $\sqrt{\frac{n}{2}(p+q)}\geq cn^{1/3}$ (use $p+q \geq p-q$). A fortiori, this implies that $\sqrt{\frac{n}2(p+q)} \gg \log n$.
Together with Lemmas~\ref{lemma:sbm:spectralnormE} and~\ref{lemma:sbm:infnormE} this shows that $\|E\|$ and $\|Eg\|_\infty$ satisfy, up to constants, the same high probability bounds as obtained through Lemma~\ref{lemma:Wignermatrices} for $\frac1{\sqrt{n}}\left(W - \ddiag(W)\right)$ and so the proof of Theorem~\ref{lem:Z2Synch:exactrecovery} applies.
\end{proof}


\section{Proof of the main results}\label{section:mainresults}


This section contains the main technical content of the paper. In particular, it provides guarantees for the Burer--Monteiro approach based solely on deterministic properties of the matrices involved. In this whole section, the cost matrix which appears in~\eqref{rank2} and~\eqref{rank2BM} is denoted by $A$, to mark that the analysis applies both the $\mathbb{Z}_2$-synchronization and to community detection, where the cost matrices are denoted respectively by $Y$ and $A^\sharp$.

%
%
%

\subsection{Partial Recovery}



Our first goal is to characterize the local optima of problem~\eqref{rank2}. In particular, we mean to show that they necessarily reveal a lot of information about the ground truth (the planted solution), under some conditions on the noise. To this end, we start by deriving the first- and second-order necessary optimality conditions of~\eqref{rank2BM}.

\begin{lemma} \label{lem:optimality}
	If $Q$ is a local maximum of \eqref{rank2BM} with cost matrix $A$, it satisfies first-order necessary optimality conditions,
	\begin{align}
	\left( \ddiag(AQQ^T) - A \right)Q = 0,
	\label{eq:firstorderconditionmatrix}
	\end{align}
	and second-order necessary optimality conditions,
	\begin{align}
	\ddiag(AQQ^T) - A\circ(QQ^T) & \succeq 0,
	\label{eq:secondorderconditionmatrix}
	\end{align}
	where $\ddiag \colon \Rnn \to \Rnn$ sets all off-diagonal entries to zero. Letting $Q = [x \ y]$, condition~\eqref{eq:firstorderconditionmatrix} also implies
	\begin{align}\label{eq:firstordercondition}
		Ax \circ y = Ay \circ x.
	\end{align}
	(In fact, they are equivalent.)
\end{lemma}
\begin{proof}
	Problem~\eqref{rank2BM} is a smooth optimization problem on a smooth manifold.
	The necessary optimality conditions correspond respectively to requiring the (projected) gradient of the cost function to vanish and the (projected) Hessian of the cost function to have appropriate curvature, see~\citep[eqs.(3.37,5.15)]{AMS08}. A derivation of exactly these conditions is done in full in~\citep{journee2010low,boumal2015staircase}, among others.
	
	We now prove~\eqref{eq:firstorderconditionmatrix} implies~\eqref{eq:firstordercondition}. Note that
	\begin{align*}
		(AQQ^T)_{ii} & = ([Ax \ Ay] Q^T)_{ii} = (Ax)_i x_i + (Ay)_i y_i.
	\end{align*}
	Hence, for all $1 \leq i \leq n$, considering the first and second columns of condition~\eqref{eq:firstorderconditionmatrix} separately, we get
	\begin{align*}
		((Ax)_i x_i + (Ay)_i)x_i & = (Ax)_i, \textrm{ and }\\
		((Ax)_i x_i + (Ay)_i)y_i & = (Ay)_i.
	\end{align*}
	Multiply the first equations by $y_i$ and the second equations by $x_i$, then subtract. It follows for all $i$:
	\begin{align*}
		0 = (Ax)_i y_i - (Ay)_i x_i,
	\end{align*}
	which can be compactly written as $Ax \circ y = Ay \circ x$.
\end{proof}


\begin{lemma}\label{lemma:Xcorrelates11T}
	Let $A = zz^T + \sigma \Delta$ for some planted signal $z \in \{ \pm 1 \}^n$, where $\Delta=\Delta^T$, $\frac1nSDP(\Delta) \leq \gamma \sqrt{n}$ and $\sigma = c\sqrt{n}$ for some $\gamma, c \geq 0$. Then, for any $Q$ satisfying the second-order condition~\eqref{eq:secondorderconditionmatrix} and corresponding $X = QQ^T$, we have
	\[
	1 \geq \frac{\langle zz^T, X \rangle}{n^2} \geq \frac{1}{2} - 2\gamma c.
	\]
	This is true in particular for all local optima, thus showing nontrivial correlation of all of those with the planted solution as soon as $\gamma c < \frac{1}{4}$. Note also that $\frac1nSDP(\Delta) \leq \|\Delta\|$ (see~\eqref{SDP:eq:inequality}).
\end{lemma}
\begin{proof}
	For any two positive semidefinite matrices $X,Y$, it holds that $\langle X, Y \rangle \geq 0$ and that $X \circ Y \succeq 0$ (Schur's product theorem). In particular, taking the inner product of the second-order optimality condition~\eqref{eq:secondorderconditionmatrix} with $X \circ (zz^T) \succeq 0$ on both sides yields the inequality
	\begin{align*}
		\inner{\ddiag(AX)}{X \circ (zz^T)} \geq \inner{A\circ X}{X \circ (zz^T)}.
	\end{align*}
	Since $\diag(X \circ (zz^T)) = \1$, the left hand side evaluates to $\trace(AX)$, so
	\begin{align*}
	\inner{A}{X} \geq \inner{A}{X \circ X \circ (zz^T)}.
	\end{align*}
	Using $A = zz^T + \sigma \Delta$ gives
	\begin{align*}
		\inner{zz^T}{X} + \sigma \inner{\Delta}{X} \geq \inner{(zz^T)\circ(zz^T)}{X \circ X} + \sigma \inner{\Delta}{X \circ X \circ (zz^T)}.
	\end{align*}
	Notice that $$\inner{(zz^T)\circ(zz^T)}{X \circ X} = \inner{\1\1^T}{X\circ X} = \inner{X}{X} = \|X\|_F^2.$$ Furthermore, both $X$ and $X \circ X \circ zz^T$ are feasible for~\eqref{SDP:fullrank}. Thus, by definition~\eqref{def:SDP:eq} and property~\eqref{SDP:eq:inequality},
	\begin{align}
	\langle zz^T, X \rangle &\geq \|X\|_F^2 - 2\sigma\operatorname{SDP}(\Delta) \geq \|X\|_F^2 - 2\gamma c n^2.
	\label{eq:oneonetransposeX}
	\end{align}
	Since $X \succeq 0$, $\rank(X) = 2$ and $\trace(X) = n$, it holds that $ \frac{n^2}{2} \leq \|X\|_F^2 \leq n^2$, so that 
	\begin{align*}
	\langle zz^T, X \rangle \geq n^2(1/2 - 2\gamma c),
	\end{align*}
	concluding the proof.
\end{proof}
We note that, in the above lemma, if $Q$ satisfies second-order conditions only approximately so that $\ddiag(AX) - A\circ C \succeq -\epsilon I_n$ with $X=QQ\transpose$, then the proof still yields $\frac{\inner{zz\transpose}{X}}{n^2} \geq \frac{1}{2} - 2\gamma c - \frac{\epsilon}{n}$.
\begin{lemma}\label{lemma:Qcorrelates1}
	(Continued from Lemma~\ref{lemma:Xcorrelates11T}.) If furthermore $\|\Delta\|\leq \gamma\sqrt{n}$ and $Q$ also satisfies the first-order condition~\eqref{eq:firstorderconditionmatrix}, then
	\[
	1 \geq \frac{\| Q^Tz \|_2}{n} \geq 1 - 8\gamma c,
	\]
	which, in particular, establishes arbitrarily strong correlation between the planted solution $z$ and any local maximum $Q$, as long as $\gamma c$ is sufficiently small.
\end{lemma}
\begin{proof}
	To prove this stronger claim, we need to show that $\|X\|_F$ is arbitrarily close to $n$ for sufficiently small $\gamma c$ (previously, we only had $\|X\|_F \geq n/\sqrt{2}$). Notice that the closer $X$ is to a rank 1 matrix, the closer its Frobenius norm is to $n$, which is in line with our endeavor. In order to get there, we will improve on the previous lemma by incorporating the first-order optimality conditions (which we did not use in the previous lemma).
	
	For ease of notation, we state the proof for $z = \1$. For the general case, apply the change of variable $A \mapsto \diag(z)A\diag(z)$, which does not require changing the assumptions about $\Delta$.
	
	The conditions on $Q$ are invariant under right-orthogonal transformation. That is, we may freely replace $Q$ by $QR$, where $R$ is an arbitrary $2\times 2$ orthogonal matrix. This allows to assume, without loss of generality, that $Q = [x \ y]$ with $\inner{x}{y} = 0$ (simply take the thin SVD of $Q = U\Sigma V^T$ and pick $R = V$). Expand the first-order condition~\eqref{eq:firstordercondition} to get
	\[
	((\1\1^T +\sigma \Delta)x)\circ y = ((\1\1^T +\sigma \Delta)y)\circ x.
	\]
	Thus,
	\[
	\langle \1,x\rangle y - \langle \1, y \rangle x = \sigma (\Delta y)\circ x - \sigma (\Delta x)\circ y.
	\]
	By taking the squared norm of both sides of the previous equation and using $\inner{x}{y}=0$, we have 
	\begin{align*}
	&\langle \1,x\rangle^2  \|y\|_2^2 + \langle \1, y \rangle^2 \| x \|_2^2 \\
	&\leq \sigma^2 \left( \|(\Delta y)\circ x\|_2 + \| (\Delta x)\circ y\|_2 \right)^2 \\
	& \leq \sigma^2 \left( \| \Delta y\|_2 + \|\Delta x\|_2 \right)^2\\
	&\leq \sigma^2 \left( 2\|\Delta\| \sqrt{n} \right)^2 \\
	&\leq c^2 n(2\gamma n)^2.
	\end{align*}
	Now, notice that $\langle \1, x\rangle^2 + \langle \1, y \rangle^2 = \langle \1\1^T, X \rangle$ and the previous lemma imply that either $\langle \1,x\rangle^2$ or $\langle \1,y\rangle^2$ (or both) is larger than or equal to $n^2(1/4 - \gamma c)$. Without loss of generality, assume it is the former:
	\[
	\langle \1,x\rangle^2 \geq n^2(1/4 - \gamma c).
	\]
	We may combine the above inequalities to get
	\[
	n^2(1/4 - \gamma c) \|y\|_2^2 \leq \langle 1,x\rangle^2 \|y\|_2^2 \leq (2\gamma c)^2 n^3,
	\]
	or equivalently
	\[
	\|y\|_2^2 \leq 4(2\gamma c)^2 n + 4\gamma c \|y\|_2^2.
	\]
	We aim to show that $\|x\|^2$ is close to $n$. Use $\|x\|_2^2+\|y\|_2^2 = n$:
	\[
	n - \|x\|_2^2 \leq (4\gamma c)^2 n + 4\gamma c n - 4\gamma c \|x\|_2^2,
	\]
	or equivalently
	\[
	(1-4\gamma c) \|x\|_2^2 \geq \left(1 - 4\gamma c - (4\gamma c)^2\right)n.
	\]
	Assuming $1 - 4\gamma c > 0$, we may divide through:
	\[
	\|x\|_2^2 \geq \frac{1 - 4\gamma c - (4\gamma c)^2}{1-4\gamma c}n.
	\]
	As targeted, if $4\gamma c$ is small enough, the right hand side gets arbitrarily close to $n$. In particular, the inequality is only informative if it is nonnegative, which requires $4\gamma c \leq \frac{\sqrt{5}-1}{2}$. By orthogonality of $x$ and $y$, we have $$\|X\|_F^2 = \|x\|_2^4 + \|y\|_2^4 \geq \|x\|_2^4,$$ so that, going back to~\eqref{eq:oneonetransposeX} and under the assumption on $4\gamma c$,
	$$
	\inner{\1\1^T}{X} \geq \|X\|_F^2 - 2\gamma c n^2 \geq \left(\left(\frac{1 - 4\gamma c - (4\gamma c)^2}{1-4\gamma c}\right)^2 - 2\gamma c\right) n^2.
	$$
	In the considered interval, the right hand side is nonnegative if and only if $0 \leq 8\gamma c \leq 1$ (a stronger condition than before). In that interval, we may extract the square root to characterize $Q^T\1$:
	\begin{align*}
	n \geq \|Q^T\1\|_2 = \sqrt{\inner{\1\1^T}{X}} \geq n\sqrt{\left(\frac{1 - 4\gamma c - (4\gamma c)^2}{1-4\gamma c}\right)^2 - 2\gamma c} \geq (1-8\gamma c)n.
	\end{align*}
	The last inequality holds by concavity of the square root term in that interval. We proved the inequality holds for $8\gamma c \leq 1$. It trivially holds otherwise as well. In passing, we note that by restricting $c$ to a smaller interval, the bound can be improved arbitrarily close to $(1-\gamma c)n$ (because less is lost in lowerbounding the concave function by an affine function).
\end{proof}
Thus, by taking $\sigma = c\sqrt{n}$ for constant $c > 0$ small enough, we have that any local maximizer $Q$ (in fact, any point satisfying both first- and second-order necessary optimality conditions) correlates very strongly with $z$.

\subsection{Exact recovery}

To establish exact recovery, we only need to show that all second-order critical points of~\eqref{rank2BM} have rank 1. Indeed, it is well-known that rank deficient second-order critical points $Q$ are global optima.
\begin{lemma}\label{lem:rankdeficientglobalopt}
	If $Q$ satisfies the second-order necessary optimality condition for~\eqref{rank2BM} and it is rank deficient (i.e., $\rank(Q) = 1$), then $Q$ is a global optimum.
\end{lemma}
\begin{proof}
	Given $Q$ feasible for~\eqref{rank2BM} with cost matrix $A$, observe that if
	\begin{align}
		S & = S(Q) = \ddiag(AQQ^T) - A
		\label{eq:S}
	\end{align}
	is positive semidefinite, then $Q$ is a global optimum. Indeed, for any feasible contender $\tilde Q$, we have (using $\diag(\tilde Q \tilde Q^T) = \1$)
	\begin{align*}
		0 \leq \innersmall{S}{\tilde Q \tilde Q^T} = \trace(AQQ^T) - \trace(A\tilde Q \tilde Q^T),
	\end{align*}
	thus showing that $\trace(Q^TAQ)$ is maximal over the search space. The matrix $S$ features prominently in~\citep{burer2002rank,burer2005local,journee2010low,wen2013orthogonality,bandeira2014tightness,boumal2015staircase,boumal2016nonconvexphase}, understandably given its strong properties as an optimality certificate. In this paper, we rely less on it in favor of different arguments.
	
	If $\rank(Q) = 1$, then $QQ^T = qq^T$ for some $q \in \{ \pm 1 \}^n$. We show that if $Q$ further satisfies the second-order condition~\eqref{eq:secondorderconditionmatrix}, then $S(Q)$ is positive semidefinite, implying optimality of $Q$. For all $u\in\Rn$, we have (using $(QQ^T)_{ij}^2 = 1$ for all $i,j$)
	\begin{align*}
		u^TSu & = \inner{\ddiag(AQQ^T) - A}{uu^T} \\
			  & = \inner{\ddiag(AQQ^T) \circ (QQ^T) - A \circ (QQ^T)}{(uu^T) \circ (QQ^T)} \\
			  & = \inner{\ddiag(AQQ^T) - A \circ (QQ^T)}{(u\circ q)(u\circ q)^T} \geq 0,
	\end{align*}
	where the inequality follows from~\eqref{eq:secondorderconditionmatrix}. This holds for all $u$, thus $S\succeq 0$.
\end{proof} 
We further show in the present context another well-known result, namely that if the perturbation $\sigma \Delta$ has good properties, then the SDP has a unique solution corresponding to the ground truth. In the Wigner setting, this lemma may be improved somewhat, see~\citep[\S2.1]{bandeira2014tightness}.
\begin{lemma}\label{lem:uniqueglobaloptexact}
	Let $A = zz^T + \sigma \Delta$ for some planted solution $z \in \{\pm 1\}^n$, with $\sigma = c\sqrt{n}$, $\Delta = \Delta^T$, $\diag(\Delta) = 0$, $\|\Delta\| \leq \gamma \sqrt{n}$ and $\|\Delta z\|_\infty \leq \gamma \sqrt{n \log n}$, for some $\gamma, c \geq 0$. If
	\begin{align*}
		\gamma c < \frac{1}{1+\sqrt{\log n}},
	\end{align*}
	then the unique global optimum of~\eqref{rank2} with cost matrix $A$ is $X = zz^T$, so that all global optima $Q$ of~\eqref{rank2BM} are of the form $Q = [z \ 0]R$, where $R$ is a $2\times 2$ orthogonal matrix.
\end{lemma}
\begin{proof}
	We show that $S = S(z) = \ddiag(Azz^T) - A$~\eqref{eq:S} is positive semidefinite with rank $n-1$, which by the argument in the proof of Lemma~\ref{lem:rankdeficientglobalopt} implies optimality of $zz^T$. Uniqueness follows from strict complementarity ($\rank(zz^T) + \rank(S(z)) = n$).
	Indeed: let $X$ be any optimum of~\eqref{rank2}; then, $\inner{S}{X} = \inner{A}{zz^T} - \inner{A}{X} = 0$. Since $S \succeq 0$ and $X \succeq 0$, $\inner{S}{X} = 0$ implies $S X = 0$. Thus, $\spann(X) \subset \ker S$. But $\ker S = \spann(z)$ since the kernel has dimension~1 and $Sz = 0$. Adding that $X \succeq 0$ and $\diag(X)=\1$, it comes that $X = zz\transpose$.
	
	Using $Sz = 0$, it remains to show that for all $u \neq 0$ such that $z^Tu = 0$, $u^TSu > 0$. We have:
	\begin{align*}
		u^TSu & = u^T\left( nI_n - zz^T + \sigma \left[ \ddiag(\Delta zz^T - \Delta) \right] \right)u \\
			& = n\|u\|_2^2 + \sigma \left[ u^T\ddiag(\Delta zz^T)u - u^T\Delta u \right] \\
			& \geq \|u\|_2^2 \left( n - \sigma \|\Delta z\|_\infty - \|\Delta \| \right) \\
			& \geq n \|u\|_2^2 \left( 1 - \gamma c \sqrt{\log n} - \gamma c  \right).
	\end{align*}
	This is indeed positive under the prescribed condition.
\end{proof}
We go through a few technical lemmas to establish the rank-one property of second-order critical points, under some conditions on the perturbation.
\begin{lemma}\label{lem: MaxCutOptimality}
	Let $QQ^T$ be any feasible point of \eqref{rank2} satisfying first- and second-order necessary optimality conditions, with $Q = [x\ y] \in \mathbb{R}^{n\times 2}$. Let $A_i$ be the $i$th row of $A$. If $\diag(A) \geq 0$, then
	\[
	\forall i, \quad \diag(AQQ^T)_i = \langle A_i, x \rangle x_i + \langle A_i, y \rangle y_i = \sqrt{ \langle A_i, x\rangle ^2 + \langle A_i, y \rangle^2 } = \|e_i^TAQ\|_2.
	\]
\end{lemma}
\begin{proof}
	Row-wise, the first-order condition \eqref{eq:firstorderconditionmatrix} reads
	\[
	\left [ \langle A_i, x \rangle, \langle A_i, y \rangle \right] = \left( \langle A_i, x\rangle x_i + \langle A_i, y \rangle y_i \right) \left[x_i, y_i \right].
	\]
	Taking the $L_2$ norm of both sides and using $x_i^2 + y_i^2 = 1$, we get
	\[
	\sqrt{ \langle A_i, x\rangle ^2 + \langle A_i, y \rangle^2 } = \left| \langle A_i, x \rangle x_i + \langle A_i, y \rangle y_i \right|.
	\]
	The second-order optimality condition~\eqref{eq:secondorderconditionmatrix} implies that (simply considering the diagonal of the positive semidefinite matrix, which must be nonnegative)
	\[
	\diag(AQQ^T) \geq \diag(A).
	\]
	That is,
	\[
	\langle A_i, x \rangle x_i + \langle A_i, y \rangle y_i  \geq A_{ii}.
	\]
	Since we assume $\diag(A) \geq 0$, the two results combine into:
	\begin{align}
	\langle A_i, x \rangle x_i + \langle A_i, y \rangle y_i = \left| \langle A_i, x \rangle x_i + \langle A_i, y \rangle y_i \right| = \sqrt{ \langle A_i, x\rangle ^2 + \langle A_i, y \rangle^2 },
	\end{align}
	concluding the proof.
\end{proof}
%
%
The following lemma bounds the largest row in $\Delta Q$. For a Wigner setting, this result is likely suboptimal, and is the bottleneck in our analysis. This is the same bottleneck that arose in both~\citep{bandeira2014tightness} and~\citep{boumal2016nonconvexphase} for the phase synchronization problem.
\begin{lemma}\label{lem:maxrowWY2}
	Let $QQ^T$ be any feasible point of \eqref{rank2} satisfying first- and second-order necessary optimality conditions with cost matrix $A = zz^T + \sigma \Delta $, $\sigma = c\sqrt{n}$, where $\Delta  = \Delta^T$ satisfies $\diag(\Delta ) = 0$, $\|\Delta \| \leq \gamma \sqrt{n}$ and $\|\Delta z\|_\infty \leq \gamma \sqrt{n \log n}$, for some $\gamma, c \geq 0$. Then, we have
	$$
	\max_i \|e_i^T\Delta Q\|_2 \leq \gamma \sqrt{n} \left( \sqrt{\log n}  + 4\sqrt{\gamma cn}\right).
	$$
\end{lemma}
\begin{proof}
	Once more, we write the proof for $z=\1$, without loss of generality.
	Let $P_\1 = \frac{1}{n} \1\1\transpose$ be the orthogonal projector to $\text{span}(\1)$ where $\1 \in\Rn$ and, analogously, let $P_{\1^\perp} = I-P_\1$ be the projector to $\text{span}(\1)^\perp$. Writing $w_i$ for the $i$th column of $\Delta$ and letting $Q = [x\ y]$, we have for all $i$:
	\begin{align*}
	\left\|e_i^T \Delta  Q \right\|_2 &= \left\| w_i^T (P_{\1} + P_{\1^\perp})Q\right\|_2\\
	&\leq \left\|w_i^T P_\1 (Q) \|_2 + \|w_i^T P_{\1^\perp} (Q) \right\|_2\\
	& \leq \frac{1}{n}\left\| w_i^T \1\1^TQ \right\|_2 + \|w_i\|_2 \left\| Q - \frac{1}{n} \1\1^TQ \right\|_F \\
	& \leq \frac{1}{n} \|\Delta \1\|_\infty \|Q^T \1\|_2 + \|w_i\|_2 \left\| Q - \frac{1}{n} \1\1^TQ \right\|_F.
	\end{align*}
	Since,
	by Lemma~\ref{lemma:Qcorrelates1},
	$\|Q^T \1\|_2 \geq n\left(1 - 8\gamma c\right)$, we may further bound the last term using
	\begin{align*}
	\left\| Q - \frac{1}{n} \1\1^TQ \right\|_F^2 & = \|Q\|_F^2 + \frac{1}{n^2} \|\1 \1^TQ\|_F^2 - \frac{2}{n} \|Q^T \1\|_2^2 \\
	& = n - \frac{1}{n} \|Q^T \1\|_2^2 \\
	& \leq n ( 1 - (1-8\gamma c)^2) \leq 16\gamma c n,
	\end{align*}
	where we used $\|Q\|_F^2 = n$.
	Combining and using $\|Q^T\1\|_2 \leq n$, we get for all $i$ that
	\begin{align*}
	\left\|e_i^T \Delta  Q \right\|_2 & \leq \gamma \sqrt{n \log n} + \|\Delta \| \sqrt{16\gamma cn} \\
	& \leq \gamma \sqrt{n} \left( \sqrt{\log n}  + 4\sqrt{\gamma cn}\right).
	\end{align*}
	This concludes the proof.
\end{proof}
The next lemma is central to our endeavor: it identifies a regime in which all second-order critical points have rank 1. The first inequality used in this proof is an important step inspired by the proof of~\citep[Thm.~3]{wen2013orthogonality}. Crucially, it is because we inject both first- and second-order optimality conditions (as opposed to only first-order conditions in~\citep{wen2013orthogonality}) that we are able to make a substantial statement via this rank-control argument. Indeed, even for the noiseless case, Theorem 3 in \citep{wen2013orthogonality} does not lead to sufficiency of $p=2$ for SDPLR.
\begin{lemma} \label{lem:master}
	Under the assumptions of Lemma~\ref{lem:maxrowWY2} and using the same notation, if
	\begin{align*}
		\gamma c < \frac{1}{9 + \sqrt{\log n} + 4\sqrt{\gamma c n}},
	\end{align*}
	then $\rank(Q) = 1$.
\end{lemma}
\begin{proof}
	From the first-order condition~\eqref{eq:firstorderconditionmatrix}, we may control the rank of $Q$ as
	\begin{align*}
	\rank(Q) & \leq \nulll\left(\ddiag(AQQ^T)-A\right) \\
	& = n - \rank\left(\ddiag(AQQ^T)-A\right) \\
	& = n - \rank\left(\ddiag(AQQ^T)-\sigma \Delta  - zz^T\right) \\
	& \leq n+1 - \rank\left(\ddiag(AQQ^T)-\sigma \Delta \right).
	\end{align*}
	To ensure $\rank(Q) = 1$, it remains to force $\rank(\ddiag(AQQ^T)-\sigma \Delta ) = n$. Since the first matrix is diagonal, this is the case in particular if
	\begin{align*}
	\min_i \diag(AQQ^T)_i > \sigma \|\Delta \|.
	\end{align*}
	This can be controlled by Lemmas~\ref{lem: MaxCutOptimality} and~\ref{lem:maxrowWY2}:
	\begin{align*}
	\min_i \|e_i^T AQ\|_2 - \sigma \|\Delta \| & = \min_i \|e_i^T zz^T Q + \sigma \cdot e_i^T \Delta  Q\|_2 - \sigma \|\Delta \| \\
	& \geq \|Q^Tz\|_2 - \sigma \cdot \max_i \|e_i^T \Delta Q\|_2 - \gamma c n \\
	& \geq n - 9\gamma c n - \gamma c n \left( \sqrt{\log n}  + 4\sqrt{\gamma cn}\right).
	\end{align*}
	Forcing the latter to be positive (as with the condition in this lemma's statement) is sufficient to imply $\rank(Q) = 1$.
	%
	%
	%
	%
\end{proof}

\begin{theorem} \label{thm:maxcutrktwo}
	Let $A = zz^T + \sigma \Delta $ for some planted solution $z \in \{\pm 1\}^n$, with $\sigma = c\sqrt{n}$, $\Delta  = \Delta^T$, $\diag(\Delta ) = 0$, $\|\Delta \| \leq \gamma \sqrt{n}$ and $\|\Delta z\|_\infty \leq \gamma \sqrt{n \log n}$, for some $\gamma, c \geq 0$. If
	\begin{align*}
		\gamma c < \frac{1}{9 + \sqrt{\log n} + 4\sqrt{\gamma c n}},
	\end{align*}
	then all second-order critical points $Q$ of~\eqref{rank2BM} with cost matrix $A$ are global optima of rank 1 such that $QQ^T = zz^T$.
%
%
	There exists a constant $k$ such that, if $\gamma c \leq k n^{-1/3}$, then the theorem applies.
\end{theorem}
\begin{proof}
	By Lemma~\ref{lem:master}, all second-order critical points $Q$ have rank 1. By Lemma~\ref{lem:rankdeficientglobalopt}, such $Q$'s are thus globally optimal. By Lemma~\ref{lem:uniqueglobaloptexact} (whose conditions are satisfied a fortiori), they all satisfy $QQ^T = zz^T$.
\end{proof}

\acks{ASB was supported by NSF grant DMS-1317308. ASB acknowledges Wotao Yin for pointing the author to the relevant reference~\citep{wen2013orthogonality} which helped motivate the start of the investigation in this paper. NB, formerly hosted by the SIERRA group at Inria and ENS, was supported by the ``Fonds Sp\'eciaux de Recherche'' (FSR) at UCLouvain and by the Chaire Havas ``Chaire Eco\-no\-mie et gestion des nouvelles don\-n\'ees'', the ERC Starting Grant SIPA and a Research in Paris grant in Paris. VV acknowledges support from the Office of Naval Research.}

\bibliography{boumal_BM}

\begin{thebibliography}{39}
\providecommand{\natexlab}[1]{#1}
\providecommand{\url}[1]{\texttt{#1}}
\expandafter\ifx\csname urlstyle\endcsname\relax
  \providecommand{\doi}[1]{doi: #1}\else
  \providecommand{\doi}{doi: \begingroup \urlstyle{rm}\Url}\fi

\bibitem[Abbe et~al.(2014)Abbe, Bandeira, Bracher, and
  Singer]{abbe2014decoding}
E.~Abbe, A.~S. Bandeira, A.~Bracher, and A.~Singer.
\newblock Decoding binary node labels from censored edge measurements: Phase
  transition and efficient recovery.
\newblock \emph{Network Science and Engineering, IEEE Transactions on},
  1\penalty0 (1):\penalty0 10--22, 2014.

\bibitem[Abbe et~al.(2016)Abbe, Bandeira, and Hall]{abbe2014exact}
E.~Abbe, A.S. Bandeira, and G.~Hall.
\newblock Exact recovery in the stochastic block model.
\newblock \emph{Information Theory, IEEE Transactions on}, 62\penalty0
  (1):\penalty0 471--487, 2016.

\bibitem[Absil et~al.(2008)Absil, Mahony, and Sepulchre]{AMS08}
P.-A. Absil, R.~Mahony, and R.~Sepulchre.
\newblock \emph{Optimization Algorithms on Matrix Manifolds}.
\newblock Princeton University Press, Princeton, NJ, 2008.
\newblock ISBN 978-0-691-13298-3.

\bibitem[Baik et~al.(2005)Baik, Ben-Arous, and P\'ech\'e]{BBP_MC_BBP_2005}
J.~Baik, G.~Ben-Arous, and S.~P\'ech\'e.
\newblock Phase transition of the largest eigenvalue for nonnull complex sample
  covariance matrices.
\newblock \emph{The Annals of Probability}, 33\penalty0 (5):\penalty0
  1643--1697, 2005.

\bibitem[Bandeira and v.~Handel(2015)]{Bandeira_NARandomMatrixBound}
A.~S. Bandeira and R.~v.~Handel.
\newblock Sharp nonasymptotic bounds on the norm of random matrices with
  independent entries.
\newblock \emph{Annals of Probability, to appear}, 2015.

\bibitem[Bandeira et~al.(2014{\natexlab{a}})Bandeira, Charikar, Singer, and
  Zhu]{Bandeira_Charikar_Singer_Zhu_Alignment}
A.~S. Bandeira, M.~Charikar, A.~Singer, and A.~Zhu.
\newblock Multireference alignment using semidefinite programming.
\newblock \emph{5th Innovations in Theoretical Computer Science (ITCS 2014)},
  2014{\natexlab{a}}.

\bibitem[Bandeira(2015{\natexlab{a}})]{Bandeira_PCC}
A.S. Bandeira.
\newblock A note on probably certifiably correct algorithms.
\newblock \emph{Available at arXiv:1509.00824 [math.OC]}, 2015{\natexlab{a}}.

\bibitem[Bandeira(2015{\natexlab{b}})]{bandeira2015laplacian}
A.S. Bandeira.
\newblock Random {L}aplacian matrices and convex relaxations.
\newblock \emph{arXiv preprint arXiv:1504.03987}, 2015{\natexlab{b}}.

\bibitem[Bandeira et~al.(2013)Bandeira, Kennedy, and
  Singer]{bandeira2013approximating}
A.S. Bandeira, C.~Kennedy, and A.~Singer.
\newblock Approximating the little {G}rothendieck problem over the orthogonal
  and unitary groups.
\newblock \emph{arXiv preprint arXiv:1308.5207}, 2013.

\bibitem[Bandeira et~al.(2014{\natexlab{b}})Bandeira, Boumal, and
  Singer]{bandeira2014tightness}
A.S. Bandeira, N.~Boumal, and A.~Singer.
\newblock Tightness of the maximum likelihood semidefinite relaxation for
  angular synchronization.
\newblock \emph{arXiv preprint arXiv:1411.3272}, 2014{\natexlab{b}}.

\bibitem[Bandeira et~al.(2015)Bandeira, Chen, and
  Singer]{bandeira2015nonuniquegames}
A.S. Bandeira, Y.~Chen, and A.~Singer.
\newblock Non-unique games over compact groups and orientation estimation in
  cryo-{EM}.
\newblock \emph{arXiv preprint arXiv:1505.03840}, 2015.

\bibitem[Barvinok(1995)]{barvinok1995problems}
A.I. Barvinok.
\newblock Problems of distance geometry and convex properties of quadratic
  maps.
\newblock \emph{Discrete \& Computational Geometry}, 13\penalty0 (1):\penalty0
  189--202, 1995.
\newblock \doi{10.1007/BF02574037}.

\bibitem[Boumal(2015)]{boumal2015staircase}
N.~Boumal.
\newblock A {R}iemannian low-rank method for optimization over semidefinite
  matrices with block-diagonal constraints.
\newblock \emph{arXiv preprint arXiv:1506.00575}, 2015.

\bibitem[Boumal(2016)]{boumal2016nonconvexphase}
N.~Boumal.
\newblock Nonconvex phase synchronization.
\newblock \emph{arXiv preprint arXiv:1601.06114}, 2016.

\bibitem[Boumal et~al.(2016)Boumal, Absil, and Cartis]{boumal2016complexity}
N.~Boumal, P.-A. Absil, and C.~Cartis.
\newblock Global rates of convergence for nonconvex optimization on manifolds.
\newblock \emph{arXiv preprint arXiv:1605.08101}, 2016.

\bibitem[Bri{\"e}t et~al.(2014)Bri{\"e}t, de~Oliveira~Filho, and
  Vallentin]{briet2014grothendieck}
J.~Bri{\"e}t, F.M. de~Oliveira~Filho, and F.~Vallentin.
\newblock Grothendieck inequalities for semidefinite programs with rank
  constraint.
\newblock \emph{Theory of Computing}, 10\penalty0 (4):\penalty0 77--105, 2014.
\newblock \doi{10.4086/toc.2014.v010a004}.

\bibitem[Burer and Monteiro(2003)]{sdplr}
S.~Burer and R.D.C. Monteiro.
\newblock A nonlinear programming algorithm for solving semidefinite programs
  via low-rank factorization.
\newblock \emph{Mathematical Programming}, 95\penalty0 (2):\penalty0 329--357,
  2003.
\newblock \doi{10.1007/s10107-002-0352-8}.

\bibitem[Burer and Monteiro(2005)]{burer2005local}
S.~Burer and R.D.C. Monteiro.
\newblock Local minima and convergence in low-rank semidefinite programming.
\newblock \emph{Mathematical Programming}, 103\penalty0 (3):\penalty0 427--444,
  2005.

\bibitem[Burer et~al.(2002)Burer, Monteiro, and Zhang]{burer2002rank}
S.~Burer, R.D.C. Monteiro, and Y.~Zhang.
\newblock Rank-two relaxation heuristics for {Max-Cut} and other binary
  quadratic programs.
\newblock \emph{SIAM Journal on Optimization}, 12\penalty0 (2):\penalty0
  503--521, 2002.

\bibitem[Decelle et~al.(2011)Decelle, Krzakala, Moore, and
  Zdeborov\'a]{Decelle_SBM}
A.~Decelle, F.~Krzakala, C.~Moore, and L.~Zdeborov\'a.
\newblock Asymptotic analysis of the stochastic block model for modular
  networks and its algorithmic applications.
\newblock \emph{Phys. Rev. E}, 84, December 2011.

\bibitem[Feige and Kilian(2001)]{Feige_Kilian_bisection_01}
U.~Feige and J.~Kilian.
\newblock Heuristics for semirandom graph problems.
\newblock \emph{Journal of Computer and System Sciences}, 63\penalty0
  (4):\penalty0 639 -- 671, 2001.

\bibitem[F\'eral and P\'ech\'e(2006)]{Feral_Peche_BBPWigner}
D.~F\'eral and S.~P\'ech\'e.
\newblock The largest eigenvalue of rank one deformation of large wigner
  matrices.
\newblock \emph{Communications in Mathematical Physics}, 272\penalty0
  (1):\penalty0 185--228, 2006.

\bibitem[Goemans and Williamson(1995)]{goemans1995maxcut}
M.X. Goemans and D.P. Williamson.
\newblock Improved approximation algorithms for maximum cut and satisfiability
  problems using semidefinite programming.
\newblock \emph{Journal of the ACM (JACM)}, 42\penalty0 (6):\penalty0
  1115--1145, 1995.
\newblock \doi{10.1145/227683.227684}.

\bibitem[Guedon and Vershynin(2014)]{Guedon_SBMgrothendieck}
O.~Guedon and R.~Vershynin.
\newblock Community detection in sparse networks via {G}rothendieck's
  inequality.
\newblock \emph{Available online at arXiv:1411.4686 [math.ST]}, 2014.

\bibitem[Hajek et~al.(2014)Hajek, Wu, and Xu]{Hajek_et_al_SBM_SDP}
B.~Hajek, Y.~Wu, and J.~Xu.
\newblock Achieving exact cluster recovery threshold via semidefinite
  programming.
\newblock \emph{Available online at arXiv:1412.6156}, 2014.

\bibitem[Javanmard et~al.(2015)Javanmard, Montanari, and
  Ricci-Tersenghi]{javanmard2015phase}
A.~Javanmard, A.~Montanari, and F.~Ricci-Tersenghi.
\newblock Phase transitions in semidefinite relaxations.
\newblock \emph{arXiv preprint arXiv:1511.08769}, 2015.

\bibitem[Journ{\'e}e et~al.(2010)Journ{\'e}e, Bach, Absil, and
  Sepulchre]{journee2010low}
M.~Journ{\'e}e, F.~Bach, P.-A. Absil, and R.~Sepulchre.
\newblock Low-rank optimization on the cone of positive semidefinite matrices.
\newblock \emph{SIAM Journal on Optimization}, 20\penalty0 (5):\penalty0
  2327--2351, 2010.
\newblock \doi{10.1137/080731359}.

\bibitem[Krivelevich and Sudakov(2003)]{Krivelevich_Sudakov_ER}
M.~Krivelevich and B.~Sudakov.
\newblock The largest eigenvalue of sparse random graphs.
\newblock \emph{Combinatorics, Probability and Computing}, 12:\penalty0 61--72,
  2003.

\bibitem[Massouli{\'e}(2014)]{Massoulie_SBM}
L.~Massouli{\'e}.
\newblock Community detection thresholds and the weak ramanujan property.
\newblock In \emph{Proceedings of the 46th Annual ACM Symposium on Theory of
  Computing}, STOC '14, pages 694--703, New York, NY, USA, 2014. ACM.
\newblock ISBN 978-1-4503-2710-7.
\newblock \doi{10.1145/2591796.2591857}.
\newblock URL \url{http://doi.acm.org/10.1145/2591796.2591857}.

\bibitem[Moitra et~al.(2015)Moitra, Perry, and Wein]{Moitra_SBMadversary}
A.~Moitra, W.~Perry, and A.~S. Wein.
\newblock How robust are reconstruction thresholds for community detection?
\newblock \emph{Available online at arXiv:1511.01473 [cs.DS]}, 2015.

\bibitem[Montanari and Sen(2015)]{Montanari_SDPdetectionSBM}
A.~Montanari and S.~Sen.
\newblock Semidefinite programs on sparse random graphs.
\newblock \emph{Available online at arXiv:1504.05910 [cs.DM]}, 2015.

\bibitem[Mossel et~al.(2014{\natexlab{a}})Mossel, Neeman, and Sly]{Mossel_SBM1}
E.~Mossel, J.~Neeman, and A.~Sly.
\newblock Stochastic block models and reconstruction.
\newblock \emph{Probability Theory and Related Fields (to appear)},
  2014{\natexlab{a}}.

\bibitem[Mossel et~al.(2014{\natexlab{b}})Mossel, Neeman, and Sly]{Mossel_SBM2}
E.~Mossel, J.~Neeman, and A.~Sly.
\newblock A proof of the block model threshold conjecture.
\newblock \emph{Available online at arXiv:1311.4115 [math.PR]}, January
  2014{\natexlab{b}}.

\bibitem[Mossel et~al.(2014{\natexlab{c}})Mossel, Neeman, and
  Sly]{Mossel_SBM3_exact}
E.~Mossel, J.~Neeman, and A.~Sly.
\newblock Consistency thresholds for the planted bisection model.
\newblock \emph{Available online at arXiv:1407.1591v2 [math.PR]}, July
  2014{\natexlab{c}}.

\bibitem[Nesterov(2004)]{nesterov2004introductory}
Y.~Nesterov.
\newblock \emph{Introductory lectures on convex optimization: A basic course},
  volume~87 of \emph{Applied optimization}.
\newblock Springer, 2004.
\newblock ISBN 978-1-4020-7553-7.

\bibitem[Pataki(1998)]{pataki1998rank}
G.~Pataki.
\newblock On the rank of extreme matrices in semidefinite programs and the
  multiplicity of optimal eigenvalues.
\newblock \emph{Mathematics of operations research}, 23\penalty0 (2):\penalty0
  339--358, 1998.
\newblock \doi{10.1287/moor.23.2.339}.

\bibitem[Shapiro(1982)]{shapiro1982rank}
A.~Shapiro.
\newblock Rank-reducibility of a symmetric matrix and sampling theory of
  minimum trace factor analysis.
\newblock \emph{Psychometrika}, 47\penalty0 (2):\penalty0 187--199, 1982.

\bibitem[Singer(2011)]{singer2010angular}
A.~Singer.
\newblock Angular synchronization by eigenvectors and semidefinite programming.
\newblock \emph{Applied and Computational Harmonic Analysis}, 30\penalty0
  (1):\penalty0 20--36, 2011.
\newblock \doi{10.1016/j.acha.2010.02.001}.

\bibitem[Wen and Yin(2013)]{wen2013orthogonality}
Z.~Wen and W.~Yin.
\newblock A feasible method for optimization with orthogonality constraints.
\newblock \emph{Mathematical Programming}, 142\penalty0 (1--2):\penalty0
  397--434, 2013.
\newblock \doi{10.1007/s10107-012-0584-1}.

\end{thebibliography}

 \appendix
 \section{Some technical steps}

\subsection{Other needed Lemmas}




\begin{lemma}\label{lemma:Wignermatrices}
 Let $W$ be a symmetric Wigner matrix whose entries are independent standard gaussian random variables and $z\in\{\pm1\}$ a fixed vector. $W-\ddiag(W)$ is the same matrix with the diagonal elements replaced by zeros. Then, the following holds for any $t\geq 0$:
 \begin{equation}\label{eq:lemma:Wspectralnorm}
  \mathrm{Prob}\left( \|W-\ddiag(W)\| \geq 2\sqrt{n} + t \right) \leq \exp\left( - \frac{t^2}4 \right),
 \end{equation}
 and
 \begin{equation}\label{eq:lemma:Winftynorm}
  \mathrm{Prob}\left( \|(W-\ddiag(W))z\|\infty \geq  \sqrt{2\log {n}+t} \right) \leq \exp\left( - \frac{t}2 \right).
 \end{equation}
\end{lemma}
\begin{proof}
\eqref{eq:lemma:Wspectralnorm} follows from combining $\EE\|W-\ddiag(W)\| \leq \EE\|W\|$ (which follows from Jensen's inequality), the well known fact that $\EE\|W\|\leq 2\sqrt{n}$, and Gaussian Concentration. 
\eqref{eq:lemma:Winftynorm} follows from standard tail bounds on gaussian random variables together with a simple union bound.
\end{proof}

Recall the definition of $E$ in~\eqref{eq:def:E}, $E$ is symmetric, its entries are independently distributed, has diagonal zero and the distribution of the entries is: for $i\neq j$ and $g_i=g_j$:
\begin{equation}
\sqrt{\frac{(p+q)n}{2}}E_{ij} = \left\{ \begin{array}{ll}
1 - p  & \text{ with probability }p \\
-p &  \text{ with probability }1-p,
\end{array}\right.
\end{equation}
and, if $g_i\neq g_j$:
\begin{equation}
\sqrt{\frac{(p+q)n}{2}}E_{ij} = \left\{ \begin{array}{ll}
1 - q  & \text{ with probability }q \\
-q &  \text{ with probability }1-q.
\end{array}\right.
\end{equation}

\begin{lemma}\label{lemma:sbm:SDPE}
There exists a constant $C$ such that, with high probability
\[
\mathrm{SDP}(E) \leq 2 + C\frac{\log d}{d^{1/10}},
\]
where $d = \frac{(p+q)n}{2} = \frac{a+b}2$ is the average degree parameter.
\end{lemma}

\begin{proof}
See Lemma~H.2. and equation (259) in~\citep{Montanari_SDPdetectionSBM}. 
\end{proof}

\begin{lemma}\label{lemma:sbm:spectralnormE}
Consider $E$ as defined above. With high probability, there exists a constant $C$ such that
\[
 \|E\| \leq 3 + C\sqrt{\frac{\log n}{\frac{n}2(p+q)}}.
\]
\end{lemma}
\begin{proof}
Set $\tilde{E} = \sqrt{\frac{(p+q)n}{2}}E$. Since the entries of $\tilde{E}$ are independent, we can use Corollary 3.12 in \citep{Bandeira_NARandomMatrixBound} (with, say, $\eps=3$) to obtain
\[
\mathrm{Prob}\left( \|\tilde{E}\| \geq 3\tilde{\sigma} + t \geq  \right) \leq n\exp\left(-\frac{t^2}{c\tilde{\sigma_\ast}^2}\right),
\]
for any $t>0$ and a universal constant $c$. Here
\[
\tilde{\sigma}  = \max_{i}\sqrt{\sum_j\EE E_{ij}^2} \leq \sqrt{\frac{n}2p(1-p) + \frac{n}2p(1-p)}, 
\]
and
\[
\tilde{\sigma_\ast} = \max_{ij}\|E_{ij}\| \leq 1.
\]
This means that, with high probability, there exists a constant $C$ such that
\[
 \|\tilde{E}\| \leq 3\sqrt{\frac{n}2p(1-p) + \frac{n}2q(1-q)} + C\sqrt{\log n} \leq 3\sqrt{\frac{(p+q)n}2} + C\sqrt{\log n},
\]
concluding the proof.
\end{proof}

\begin{lemma}\label{lemma:sbm:infnormE}
 Consider $E$ as defined above. With high probability, there exists a constant $C$ such that
 \[
  \|Eg\|_\infty \leq C\sqrt{\log n} + C\frac{\log n}{\sqrt{\frac{n}2(p+q)}}
 \]
\end{lemma}
\begin{proof}
Set $\tilde{E} = \sqrt{\frac{(p+q)n}{2}}E$. For each $i\in [n]$, we have
\[
 (\tilde{E}g)_i = \sum_{j=1}^{\frac{n}2-1}\xi_j - \sum_{j=1}^{\frac{n}2}\xi'_j,
\]
where $\xi_j$ are independent random variables taking the value $1-p$ with probability $p$ and the value $-p$ with probability $1-p$; $\xi'_j$ are independent random variables (and independent to the random variables $\xi'_{j'}$) taking the value $1-q$ with probability $q$ and the value $-q$ with probability $1-q$. It is easy to see that
\[
 \sum_{j=1}^{\frac{n}2-1}\EE\xi_j^2 - \sum_{j=1}^{\frac{n}2}\EE\left(\xi'_j\right)^2 = \frac12(\frac{n}2-1)p(1-p) + \frac{n}2q(1-q) \leq \frac{n}2(p(1-p)+q(1-q)) ,
\]
and that the summands are almost surely bounded by $1$. This means that we can use Bernstein's inequality to get
\[
 \mathrm{Prob}\left( (\tilde{E}g)_i > t \right) \leq \exp\left( - \frac{\frac12t^2}{\frac{n}2(p(1-p)+q(1-q)) + \frac13t} \right).
\]
A union bound gives
\[
 \mathrm{Prob}\left( \|\tilde{E}g\|_\infty > t \right) \leq 2n\exp\left( - \frac{\frac12t^2}{\frac{n}2(p(1-p)+q(1-q)) + \frac13t} \right),
\]
which means that, with high probability,
\[
 \|\tilde{E}g\|_\infty \lesssim \log n + \sqrt{\frac{n}2(p(1-p)+q(1-q))}\sqrt{\log n} \leq \log n + \sqrt{\frac{n}2(p+q)}\sqrt{\log n},
\]
concluding the proof.
\end{proof}

\begin{remark}
It is worth noting that the quantities $\|E\|$ and $\|Eg\|_\infty$ are tightly connected to the control of the spectrum of $\Gamma_{SBM}$ in~\citep[Definition 4.8]{bandeira2015laplacian}. 
\end{remark}

%

%
%
%
%
%
%
%
%
%
%
%
%
%
%
%

\end{document}